\documentclass[12pt]{article}
\usepackage{amscd, amssymb, latexsym, amsmath}
\newtheorem{thm}{Theorem}

\newtheorem{pro}{Proposition}

\newtheorem{dfn}{Definition}

\numberwithin{thm}{section} \numberwithin{cor}{section}
\numberwithin{pro}{section} \numberwithin{dfn}{section}
\numberwithin{lem}{section}
\numberwithin{rem}{section}\numberwithin{equation}{section}
\newenvironment{proof}
{\noindent {\em Proof. }} {\hfill $\Box$}
\newcommand{\f}{\frac}
\newcommand{\no}{\noindent}

\newcommand{\sqp}{\sqrt{p}}
\newcommand{\sqq}{\sqrt{q}}
\newcommand{\ept}{e^{ip\theta}}
\newcommand{\eqt}{e^{-iq\theta}}

\newcommand{\ppq}{\sqrt{\f{p}{p+q}}}
\newcommand{\qpq}{\sqrt{\f{q}{p+q}}}
\newcommand{\ba}{\begin{array}}
\newcommand{\ea}{\end{array}}

\newcommand{\R}{\mathbb R}
\newcommand{\C}{\mathbb C}
\newcommand{\mtc}{\sqrt{\f{-t}{c}}}
\newcommand{\ptc}{\sqrt{\f{t}{c}}}

\newcommand{\dst}{d|\!|S_t|\!|}
\newcommand{\dszero}{d|\!|S_0|\!|}
\newcommand{\deet}{d|\!|E_t|\!|}
\newcommand{\dezero}{d|\!|E_0|\!|}

\begin{document}
\title
{Hamiltonian Stationary Shrinkers and Expanders for Lagrangian Mean Curvature Flows}
\author{Yng-Ing Lee* and Mu-Tao Wang**}

\date{May 20, 2007, revised June 29, 2007}
\maketitle \leftline{*Department of Mathematics and Taida
Institute of Mathematical Sciences,}
\leftline{\text{ }\,National
Taiwan University, Taipei, Taiwan}
 \leftline{\text{ }\,National
Center for Theoretical Sciences, Taipei Office}
 \centerline{email: yilee@math.ntu.edu.tw}
\leftline{**Department of Mathematics, Columbia University,  New
York, NY 10027,USA} \leftline{\text{  }\,\,\,Taida Institute of Mathematical Sciences, Taipei, Taiwan}\centerline{email: mtwang@math.columbia.edu}

\begin{abstract}
We construct examples of shrinkers and expanders for Lagrangian
mean curvature flows. These examples are Hamiltonian stationary
and asymptotic to the union of two  Hamiltonian stationary cones
found by Schoen and Wolfson in \cite{sw}. The Schoen-Wolfson cones
$C_{p,q}$ are obstructions to the existence problems of special
Lagrangians or Lagrangian minimal surfaces in the variational
approach. It is known that these cone singularities cannot be
resolved by any smooth oriented Lagrangian submanifolds. The
shrinkers and expanders that we found can be glued together to
yield solutions of the Brakke motion-a weak formulation of the
mean curvature flow. For any coprime pair $(p,q)$ other than
$(2,1)$, we construct such a solution that resolves any single
Schoen-Wolfson cone $C_{p,q}$. This thus provides an evidence to
Schoen-Wolfson's conjecture that the $(2,1)$ cone is the only
area-minimizing cone. Higher dimensional generalizations are also
obtained.


\end{abstract}
\section{Introduction}

The existence of special Lagrangians in  Calabi-Yau manifolds
received much attention recently due to the critical role it plays
in the T-duality formulation of Mirror symmetry of
Strominger-Yau-Zaslow \cite{syz}.  Schoen and Wolfson took up the
variational approach of constructing special Lagrangians by
minimizing volumes in suitable Lagrangian classes. They discovered
non-flat Lagrangian cones that are Hamiltonian stationary
\cite{sw}. The existence of special Lagrangians can be established
once these cone singularities are excluded. However, these
singularities do occur in the Lagrangian minimizers in some K-3
surfaces, see \cite{wo}. Another potential approach to the
construction of special Lagrangians is the mean curvature flow- as
the gradient flow of the volume functional. However, the long-time
existence of such flows can only be verified in some special
cases, see for example \cite{sm}, \cite{sw}, \cite{wa1}, and
\cite{wa2}. In this article, we construct special weak solutions
of the Lagrangian mean curvature flows and show that the union of
two Schoen-Wolfson cones can be resolved by these flows.

Our ambient space is always the complex Euclidean space $\C^n$
with coordinates $z^i=x^i+\sqrt{-1} y^i$, the standard symplectic
form $\omega= \sum_{i=1}^n dx^i \wedge dy^i$, and the standard
almost complex structure $J$ with $J(\frac{\partial}{\partial
x^i})=\frac{\partial}{\partial y^i}$. On a Lagrangian submanifold
$\Sigma$, the mean curvature vector ${H}$ is given by
\begin{equation}\label{mcv} {H}=J\nabla \beta\end{equation} where
$\beta$ is the Lagrangian angle and $\nabla$ is the gradient on
$\Sigma$. By the first variation formula, the mean curvature
vector points to the direction where the volume is decreased most
rapidly. In this case, $\beta$ can be defined by the relation that
\[*_\Sigma (dz^1\wedge\cdots\wedge dz^n)=e^{i\beta}\] where $*_\Sigma$ is the Hodge *-star operator on $\Sigma$. We recall

\begin{dfn}
A Lagrangian submanifold $\Sigma$ is called Hamiltonian stationary
if the Lagrangian angle is harmonic. i.e. $\Delta \beta=0$ where
$\Delta$ is the Laplace operator on $\Sigma$. $\Sigma$ is a
special Lagrangian if $\beta$ is a constant function.
\end{dfn}

A Hamiltonian stationary Lagrangian submanifold is a critical
point of the volume functional among all Hamiltonian deformations
and a special Lagrangian is a volume minimizer in its homology
class.

 As the special Lagrangians are volume minimizers, it is thus natural to use the mean curvature flow in the construction of special Lagrangians. Equation (\ref{mcv}) implies
that the mean curvature flow is a Lagrangian deformation, i.e. a
Lagrangian submanifold remains Lagrangian along the mean curvature
flow. In a geometric flow, the singularity often models on a
soliton solution. In the case of mean curvature flows , one type
of soliton solutions of particular interest are those moved by
scaling in the Euclidean space. We recall:

\begin{dfn}
A Lagrangian submanifold in the Euclidean space is called a
self-similar solution if

\[F^\perp=2cH\] for some constant $c$, where $F^\perp$ is normal projection of the position vector $F$ in the Euclidean space and $H$ is the mean curvature vector. It is called a self-shrinker if $c<0$  and self-expander if $c>0$.
\end{dfn}

It is not hard to see that if $F$ is a self-similar solution, then
$F_t$ defined by $F_t=\sqrt{\frac{t}{c}} F$ is moved by the mean
curvature flow. By Huisken's monotonicity formula \cite{hu}, any
central blow up of a finite-time singularity of the mean curvature
flow is a self-similar solution.  In this article, we find
Hamiltonian stationary self-shrinkers and self-expanders of the
Lagrangian mean curvature flow that are asymptotic to the union of
two Schoen-Wolfson cones. Altogether they form a Brakke flow (see
\S \ref{brakke} ) which is a weak formulation of the mean curvature
flow proposed by Brakke in \cite{br}. To be more precise, we prove:

\begin{thm}
 For each Schoen-Wolfson cone $C_{pq}$ , there exists a corresponding
 cone $C_{pq}'$ and a solution
$V_t, -\infty<t<\infty$ of the Brakke motion without mass loss so
that $V_t, t<0$ is a  smooth Hamiltonian stationary self-shrinker
and $V_t, t>0$ is a smooth Hamiltonian stationary self-expander.
Moreover, $V_t$ approach the union $C_{pq}\cup C_{pq}'$ as
$t\rightarrow 0$ from either direction.

\end{thm}

\begin{dfn}
We call such a solution  $V_t, -\infty<t<\infty$, a Hamiltonian
stationary self-similar eternal Brakke motion.

\end{dfn}

Without loss of generality, we can assume that $p>q$. When $q>1$,
we show that a single Schoen-Wolfson cone can be resolved by
self-similar Brakke motion.

\begin{thm}
 For any Schoen-Wolfson cone $C_{p,q}$ with $q>1$, there exists a Hamiltonian stationary
 self-similar eternal Brakke motion $V_t$ such that
$V_t$ approaches $C_{p,q}$ as $t\rightarrow 0$ from either
direction.
\end{thm}

Schoen-Wolfson show that $C_{p,q}$ is stable only if $p-q=1$ and they conjecture that only the $C_{2,1}$ cone is area-minimizing. The first author were informed by R. Schoen that this is the first time when $C_{2,1}$ can be distinguished from all other stable  $(p,q)$ cones.

We remark that self-similar solutions of Lagrangian mean curvature
flows were constructed by Anciaux \cite{an} using a different
ansatz. His examples approach special Lagrangian cones while ours
approach Hamiltonian stationary cones. Haskins \cite{ha3} also
observed these solutions are Hamiltonian stationary based on a
Hamiltonian formulation similar to the one used by Harvey and
Lawson \cite{hl} in their construction of examples of special
Lagrangians. Special Lagrangians and Hamiltonian stationary
Lagrangians are constructed by many authors.

Our theorem is analogous to the Feldman-Ilmanen-Knopf \cite{fik}
gluing construction for the K\"ahler-Ricci flows. Unlike the mean curvature flow, a notion of weak solutions
of Ricci flow has not yet been established.

This article is organized as the follows. Our examples are
presented in \S2 and the formulation of Brakke motion is recalled
in \S3. In section \S4 and \S5, we prove theorem 1.1 and theorem
1.2, respectively. Higher dimensional examples are presented in
\S6.

Both authors thank Mark Haskins for enlightening conversations and
comments on this subject. They would like to thank Dominic Joyce
for referring them to the articles \cite{jo1} and \cite{jo2} and
helpful remarks. The first author owes her gratitude to Tom
Ilmanen for discussions that lead to the finding of the
corresponding  cones and she thanks the hospitality of Ilmanen and
Institute for Mathematical Research at ETH during her visit. The
second author wishes to thank the support of the Taida Institute
for Mathematical Sciences during the preparation of this article.
The first author is supported by Taiwan NSC grant 95-2115-M-002.
The second author is supported by NSF grant DMS0605115 and a Sloan
research fellowship.

\section{Examples of self-similar solutions}
\subsection{Schoen-Wolfson cones}
\no Let $p$ and $q$ be two co-prime positive integers. The close and embedded curve
$$\gamma_{pq} (\theta)=(\qpq\:\ept \,,\,i\ppq\:\eqt), \hskip.6cm 0\leq \theta < 2\pi$$
is Legendrian and Hamiltonian stationary in $S^{3}$.
The cone over $\gamma_{pq}(\theta)$
is Lagrangian and Hamiltonian stationary and is denoted by
$C_{pq}$. It is stable if and only if $|p-q|=1$. As such
properties are invariant under $U(2)$, the cone over any $U(2)$
image of $\gamma_{pq}$ is again Lagrangian and Hamiltonian
stationary.
 These are possible cone singularities for the Lagrangian
minimizers studied in \cite{sw}.

\subsection{Self-shrinkers and self-expanders}

We take the same ansatz as Schoen-Wolfson and consider the
surfaces
$$F(\mu, \theta)=(z_{1}(\mu)\:\ept\,,\,z_{2}(\mu)\:\eqt), $$ where
 $ 0\leq \theta <  2\pi, \hskip.2cm \mu \in \R,$ and $z_{1}(\mu)$ and $z_{2}(\mu)$
  are curves in the complex
 plane. A direct computation shows that a sufficient condition for $F(\mu, \theta)$
 to be
  Lagrangian is that $p|z_{1}(\mu)|^{2}-q|z_{2}(\mu)|^{2}$ is a
  constant. One can further investigate the condition for $F(\mu, \theta)$  to be Hamiltonian
  stationary, and the condition for $F(\mu, \theta)$ to be self-similar. We will not explore
  the general situation here. Instead we show directly in the following that
  the surface
\[F(\mu, \theta)=(\cosh{\mu}\sqq\ept\,,\,i\sinh{\mu}\sqp\eqt), \]  where
 $ 0\leq \theta < 2\pi$ and $ \mu \in \R,$
 is  Hamiltonian stationary and self-similar.
We compute
 \[\f{\partial F}{\partial \mu}=(\sinh{\mu}\sqq\:\ept\,,\,i\cosh{\mu}\sqp\:\eqt), \]
and \[\f{\partial F}{\partial
\theta}=\sqrt{pq}(i\cosh{\mu}\sqp\:\ept\,,\,\sinh{\mu}\sqq\:\eqt). \]

 It is
easy to check that $\langle J\f{\partial F}{\partial
\mu},\f{\partial F}{\partial \theta}\rangle =0$, and thus the
surface is Lagrangian. The components of the induced metric on the
surface are
\[g_{11}=\left|\f{\partial F}{\partial \mu}\right|^{2}={p}\cosh^{2}{\mu}+{q}\sinh^{2}{\mu},\]
\[g_{22}=\left|\f{\partial F}{\partial \theta}\right|^{2}=pq(p\cosh^{2}{\mu}+
q\sinh^{2}{\mu}),\] and $g_{12}=0$. Therefore the area form is
given by
\[\sqrt{pq}(p\cosh^2\mu+q\sinh^2\mu)d\mu
d\theta\]

A simple calculation shows that the Lagrangian angle
$\beta=(p-q)\theta$. Thus $\Delta_{g} \beta=0$, it follows that
the surface is Hamiltonian stationary. On a Lagrangian
submanifold, we have the mean curvature vector $H=J\nabla \beta$.
Since $\beta$ depends only on $\theta$,
$$H=\f{1}{g_{22}}J\f{\partial\beta}{\partial \theta}\f{\partial
F}{\partial \theta}= \f{p-q}{g_{22}}J\f{\partial
F}{\partial \theta}. $$

To calculate $F^\perp$ we note that the normal bundle of the
surface is spanned by  $J\f{\partial F}{\partial \mu}$ and $J\f{\partial F}{\partial \theta}$. We compute
\[\begin{split}\langle F,J\f{\partial F}{\partial \mu}\rangle&={\it Re}\;(-iq\cosh{\mu}\sinh{\mu}
 -ip\sinh{\mu}\cosh{\mu})=0
\end{split}\]
and
\[\begin{split}\langle F,J\f{\partial F}{\partial \theta}\rangle &=\sqrt{pq}{\it Re}\;(-\cosh^{2}{\mu}\sqrt{pq}
+\sinh^{2}{\mu}\sqrt{pq})=-pq.
\end{split}\]
Hence \begin{equation}F^{\perp}=\f{-pq}{g_{22}}\:J\f{\partial
F}{\partial \theta} =-\f{pq}{p-q}H\end{equation} and $F$ is a
self-similar solution.  We summarize the calculations in this
section in the following proposition.
\begin{pro}
If $p>q$ are two co-prime positive integers, then
$$S(\mu, \theta)= (\cosh \mu\sqq e^{ip\theta}\,,\,i\sinh \mu \sqp e^{-iq\theta}), $$  where
 $ 0\leq \theta < 2\pi$ and $ \mu \in \R,$
 is a Hamiltonian stationary shrinker and
$$E(\mu,\theta)= (\sinh \mu \sqq e^{ip\theta} \,,\,i\cosh \mu \sqp e^{-iq\theta}), $$ where
 $ 0\leq \theta< 2\pi$ and $ \mu \in \R,$
 is a Hamiltonian stationary expander. $S$ satisfies $F^\perp=-2cH$ while $E$ satisfies
$F^\perp=2c H$, where $c=\frac{pq}{2(p-q)}$.
\end{pro}

We notice that $E$ can be obtained by switching $p$ and $q$ in the
expression for $S$, taking bar, and multiplying by
$\begin{bmatrix} 0&i\\i&0\end{bmatrix}$. As $\mu\rightarrow
+\infty$, both $S$ and $E$ approach the Schoen-Wolfson cone
$C_{pq}$ over the curve $\gamma_{pq}$.

\subsection{Asymptotics of the flow}

 By the remark in the introduction, we have $\mtc S$ for $t<0$ is a smooth solution
 of the mean curvature flow, so is  $\ptc E$ for $t>0$.

\begin{pro}
If $p>q$ are two co-prime positive integers, then
$$S_t(\mu, \theta)= \mtc (\cosh \mu\sqq e^{ip\theta}\,,\,i\sinh \mu \sqp e^{-iq\theta}), $$  for $t<0$ is a smooth
solution of the mean curvature flow and so is

$$E_t(\mu,\theta)= \ptc(\sinh \mu \sqq e^{ip\theta} \,,\,i\cosh \mu \sqp e^{-iq\theta}), $$  for $t>0$.
\end{pro}
Denote by $h(S_t)$ the mean curvature vector of $S_t$ and
$d|\!|S_t|\!|$ the area element of $S_t$, then

\begin{equation}\label{stnorm}\left|S_t\right|^2=\left(\f{-t}{c}\right)(q\cosh^2\mu+p\sinh^2\mu),\end{equation}
\begin{equation}\label{hstnormsq}\left|h(S_t)\right|^2=\left(\f{c}{-t}\right)\f{(p-q)^2}{pq}\f{1}{p\cosh^2\mu+q\sinh^2\mu},\end{equation}

\begin{equation}\label{dst}d|\!|S_t|\!|=\left(\f{-t}{c}\right)\sqrt{pq}(p\cosh^2\mu+q\sinh^2\mu)d\mu
d\theta.\end{equation}

For positive co-prime integers $p$ and $q$ with $p>q$, we define
\[\begin{split}
C_{++}(y,\theta)&=(y\sqq\:e^{ip\theta}\,,\,iy\sqp\:e^{-iq\theta}),\\
C_{+-}(y,\theta)&=(y\sqq\:e^{ip\theta}\,,\,-iy\sqp\:e^{-iq\theta}),\\
C_{-+}(y,\theta)&=(-y\sqq\:e^{ip\theta}\,,\,iy\sqp\:e^{-iq\theta}),\text{and} \\
C_{--}(y,\theta)&=(-y\sqq\:e^{ip\theta}\,,\,-iy\sqp\:e^{-iq\theta}),
\end{split}\]
for $y\geq 0$ and $ 0\leq \theta < 2\pi$. Here $C_{++}=C_{pq}$.

Note that $S_t$, as $t\rightarrow 0^-$, approaches $C_{++}\cup
C_{+-}$ while $E_t$, as $t\rightarrow 0^+$, approaches $C_{++}\cup
C_{-+}$. The asymptotic cones of $S_t$ and $E_t$ do not match
unless $p$ and $q$ are both odd. In other cases, we can modify
$S_t$ and $E_t$ so their asymptotic cones agree at $t=0$. This
allows us to construct a Brakke flow $F_t$ that is a Hamiltonian
stationary self-shrinker for $t<0$, a Hamiltonian stationary
self-expander for $t>0$, and a pair of cones at $t=0$ in all
cases.

\section{Brakke motion}\label{brakke}

 A family of varifolds
$V_{t}$ is said to form a solution of the Brakke motion \cite{br}
if
\begin{equation}\label{brakkeineq}\bar{D}|\!|V_{t}|\!|(\phi)\leq
\delta (V_{t},\phi)(h(V_{t}))\end{equation}
 for each $\phi \in
C^{1}_{0}(\R^{n})$ with $\phi \geq 0$, where
$\bar{D}|\!|V_{t}|\!|(\phi)$ is the upper derivative defined by $
\overline{\lim}_{t_{1}\rightarrow t }
\f{|\!|V_{t_{1}}|\!|(\phi)-|\!|V_{t}|\!|(\phi)} {t_{1}-t}$ and
$h(V_{t})$ is the generalized mean curvature vector of $V_{t}$. In
the setting of this paper,
$$ \delta (V_t,\phi)(h(V_t))=-\int \phi |h(V_t)|^{2} d|\!|V_t|\!|
+\int D\phi\cdot h(V_t)d|\!|V_t|\!|.$$

In our case, the singularity happens at the $t=0$ slice. We
formulate the following proposition as a criterion to check the
solutions of Brakke motion in this case.

\begin{pro}\label{mainpro}

Suppose the varifold $V_t$, $a <t<b$ forms a smooth mean curvature
flow in $\R^n$ except at $t=c\in (a, b)$ and $|\!|V_t|\!|$
converges in Radon measure to $|\!|V_c|\!|$ as $t\rightarrow c$.
If $ \lim_{t\rightarrow c^-} \f{d}{dt}|\!|V_t|\!|(\phi)$ and
$\lim_{t\rightarrow c^+} \f{d}{dt}|\!|V_t|\!|(\phi)$ are both
either finite or $-\infty$ and

\begin{equation}\label{assumption}\lim_{t\rightarrow
c^\pm} \f{d}{dt}|\!|V_t|\!|(\phi)\leq \delta (V_0,
\phi)(h(V_0))\end{equation} for any $\phi\in C^1_0(\R^n)$ then
$V_t$ forms a solution of the Brakke motion.
\end{pro}

\begin{proof}
Since $|\!|V_t|\!|(\phi)$ is continuous on the interval $(a, c]$
and differentiable on $(a, c)$. By the mean value theorem, we have

\[\lim_{t\rightarrow c^-}\frac{|\!|V_t|\!|(\phi)-|\!|V_c|\!|(\phi)}{t-c}=\lim_{t\rightarrow
c^-}\f{d}{dt} |\!|V_t|\!|(\phi).\]
The case for $t\rightarrow c^+$ can be treated similarly.
Therefore the assumption (\ref{assumption}) implies
that (\ref{brakkeineq}) holds.

\end{proof}

\section{Proof of Theorem 1.1}
The proof of Theorem 1.1 is divided into three cases according to the parities of $p$ and $q$. In each case, we show
that (\ref{assumption}) holds with equality.

Before going into the details of the proof, we first make some
observations on the asymptotic cones. Consider shrinkers of the
form

\[\{\mtc(x_1\sqrt{q}e^{ip\theta},x_2\sqrt{p} e^{-iq\theta}):
|x_1|^2-|x_2|^2=1, t<0\}\] and expanders

\[\{\ptc(x_1\sqrt{q}e^{ip\theta},x_2\sqrt{p} e^{-iq\theta}):
|x_1|^2-|x_2|^2=-1,t>0\}.\]

As $t\rightarrow 0$ both of them converge to $C_{++}\cup C_{+-}
\cup C_{-+} \cup C_{--}$. By shifting $\theta$ to $\theta+\pi$, it
is easy to see that

(i) When both $p$ and $q$ are odd, $C_{++}=C_{--}$ and
$C_{+-}=C_{-+}$.

(ii) When $p$ is odd and $q$ is even, $C_{++}=C_{-+}$ and
$C_{+-}=C_{--}$.

(iii) When $p$ is even and $q$ is odd, $C_{++}=C_{+-}$ and
$C_{-+}=C_{--}$.

That is, the shrinkers and expanders converge to the double of two
cones. In the following, we manage to arrange $S_t$ and $E_t$ so
that they converge to a single copy of the two cones.  More
precisely, when $p$ and $q$ are both odd, the asymptotic cones are
$C_{++}\cup C_{+-}$. When $p$ is odd and $q$ is even, the
asymptotic cones are $C_{++}\cup C_{+-}$. When $p$ is even and $q$
is odd, the asymptotic cones are $C_{++}\cup C_{-+}$.
\subsection{ Case 1: both $p$ and $q$ are odd}

We start with $t<0$.  By change of variable $y=\mtc \sinh\mu$, it
is not hard to see $S_t$ as $t\rightarrow 0^-$ converges to the
varifold $S_0$ defined by

\[S_0(y,\theta)=(|y|\sqq\,\ept, iy\sqp\,\eqt), y\in \R, 0\leq \theta <2\pi.\]

The norm square of $S_0$ is given by

\begin{equation}\label{s0norm}\left|S_0\right|^2=y^2(p+q).\end{equation}

The norm square of generalized mean curvature vector and the area
element of $S_0$ are given by:

\begin{equation}\label{hs0normsq}\left|h(S_0)\right|^2=\f{(p-q)^2}{pq(p+q)}\f{1}{y^2},\end{equation}

and \begin{equation}\label{dszero}\dszero=|y|\sqrt{pq}(p+q)dy
d\theta.\end{equation}

Since for a smooth mean curvature flow, we have

\[\f{d}{dt}|\!|S_t|\!|(\phi)=\delta (S_t,\phi)(h(S_t))=-\int \phi
|h(S_t)|^{2} \dst +\int D\phi \cdot h(S_t) \dst.\]

To apply Proposition \ref{mainpro}, it suffices to show
\begin{equation}\label{mainid}\begin{split}&\lim_{t\rightarrow 0^-}-\int \phi \,|h(S_t)|^{2}
\dst
+\int D\phi \cdot h(S_t) \dst\\
&=-\int \phi \,|h(S_0)|^{2} \dszero +\int D\phi \cdot h(S_0)
\dszero\end{split}\end{equation} and the limit is either finite or
$-\infty$. From (\ref{hstnormsq}) and (\ref{dst}), we obtain

\begin{equation}\label{hstsqdst}\left|h(S_t)\right|^2 \dst=\f{(p-q)^2}{\sqrt{pq}} d\mu
d\theta\end{equation} and

\begin{equation}\label{hstdst}\left|h(S_t)\right|
\dst=(p-q)\sqrt{\left(\f{-t}{c}\right)(p\cosh^2\mu+q\sinh^2\mu)}
d\mu d\theta.\end{equation}

We first show that $\int D\phi \cdot h(S_t) \dst$ is always
finite. By (\ref{hstdst}), this integral is bounded above by

\begin{equation}\label{dphihst}\begin{split}&\int \left|D\phi\right| \left|h(S_t)\right| \dst \\
&=(p-q)\int
\left|D\phi(S_t(\mu,\theta))\right|\sqrt{\left(\f{-t}{c}\right)(p\cosh^2\mu+q\sinh^2\mu)}
d\mu d\theta.\end{split}\end{equation}

Suppose $\phi$ vanishes outside $B(0;R)$ and recall the expression
(\ref{stnorm}) for $|S_t|$, we see the integral is supported in
the domain \[\left(\f{-t}{c}\right)(q\cosh^2\mu+p\sinh^2\mu)\leq
R^2, 0\leq \theta<2\pi.\]

Thus (\ref{dphihst}) is bounded above by

\begin{equation}\label{C1int}C_1 \int_{\left(\f{-t}{c}\right)(q\cosh^2\mu+p\sinh^2\mu)\leq
R^2 } \sqrt{\left(\f{-t}{c}\right) (p\cosh^2\mu+q\sinh^2\mu)}
d\mu\end{equation} for some constant $C_1>0$ depending on the
upper bound of $|D\phi|$.

Consider the change of variable $y=\mtc \sinh \mu$, we have
$dy=\mtc \cosh\mu d\mu$ and
$\mtc\cosh\mu=y\f{\cosh\mu}{\sinh\mu}$, (\ref{C1int}) becomes

\[C_1 \int_{|y|\leq \sqrt{\f{R^2+(\f{t}{c})q}{p+q}} }\sqrt{p+q\tanh^2 \mu} dy\] which is finite
as $\tanh^2\mu\leq 1$.

Next we claim the limit $ \lim_{t\rightarrow 0^-}-\int \phi
|h(S_t)|^{2} \dst$ is finite if $ \phi(0)=0$ and $-\infty$ if
$\phi(0)\not= 0$.

By (\ref{hstsqdst}),
\[\int \phi \,|h(S_t)|^{2} \dst=\int \phi(S_t(\mu, \theta))\f{(p-q)^2}{\sqrt{pq}}d\mu d\theta.\]

When $\phi(0)=0$, we may assume $\phi$ is supported in $B(0;R)$
and $\phi(S_t(\mu, \theta))\leq C_2\left|S_t(\mu, \theta)\right|$ for some $C_2>0$, therefore

\[\int \phi(S_t(\mu, \theta)) d\mu d\theta\leq C_3\int_{\left(\f{-t}{c}\right)(q\cosh^2\mu+p\sinh^2\mu)\leq R^2 } \sqrt{\left(\f{-t}{c}\right)(q\cosh^2\mu+p\sinh^2\mu)} d\mu\] for some $C_3>0$.

This is similar to (\ref{dphihst}) and can be shown to be finite
by the change of variable $y=\mtc\sinh\mu$.

On the other hand, when $\phi(0)>0$, we may assume $\phi(0)\geq
C_4>0 $ on $B(0;\epsilon)$, thus

\[\int \phi(S_t(\mu, \theta))d\mu d\theta\geq 2\pi C_4\int_{\left(\f{-t}{c}\right)(q\cosh^2\mu+p\sinh^2\mu)
\leq \epsilon^2 } d\mu=2\pi C_4\int_{|y|\leq
\sqrt{\f{\epsilon^2+(\f{t}{c})q}{p+q}}
}\f{1}{\sqrt{y^2+(\f{-t}{c})}} dy\] which tends to $\infty$ as $t\rightarrow 0^-$ by observing
$\f{1}{\sqrt{y^2+(\f{-t}{c})}}\geq \f{1}{|y|+\sqrt{\f{-t}{c}}}$.

Equations (\ref{s0norm}), (\ref{hs0normsq}), and (\ref{dszero})
imply that $\int\phi|h(S_0)|^2 d|\!|S_0|\!|$ is finite if
$\phi(0)=0$ and $-\infty$ if $\phi(0)>0$. Now (\ref{mainid})
follows from the change of
variable $y=\mtc\sinh\mu$, the fact that $h(S_t)\rightarrow h(S_0)$ , and the dominant convergence theorem.

For $t>0$, we consider $E_t$ with
\begin{equation}\label{etnorm}\left|E_t\right|^2=\left(\f{t}{c}\right)(p\cosh^2\mu+q\sinh^2\mu),\end{equation}
\begin{equation}\label{hetnormsq}\left|h(E_t)\right|^2=\left(\f{c}{t}\right)\f{(p-q)^2}{pq}\f{1}{q\cosh^2\mu+p\sinh^2\mu},\end{equation}

\begin{equation}\label{det}d|\!|E_t|\!|=\left(\f{t}{c}\right)\sqrt{pq}(q\cosh^2\mu+p\sinh^2\mu)d\mu
d\theta.\end{equation}

As $t\rightarrow 0^+$, $E_t$ converges to the varifold $E_0$
defined by \[E_0(y,\theta)= (y \sqq e^{ip\theta} \,,\,i|y| \sqp
e^{-iq\theta}), y\in \R, 0\leq \theta<2\pi.\]

$E_0$ coincides with $S_0$ by the change of variable

\begin{equation}\label{cov} E_{0}(y, \theta)=S_{0} (y, \theta+ \arg y).\end{equation} when
$p$ and $q$ are both odd (note that $\arg y=0$ or $\pi$).

The identity
\begin{equation}\label{mainid2}\begin{split}&\lim_{t\rightarrow 0^+}-\int \phi \,|h(E_t)|^{2}
\deet
+\int D\phi \cdot h(E_t) \deet\\
&=-\int \phi \,|h(E_0)|^{2} \dezero +\int D\phi \cdot h(E_0)
\dezero\end{split}\end{equation} can be checked similarly.

\subsection{ Case 2:  $p$ odd and $q$ even}
In this case, for $t<0$, $V_{t}$ is defined to be $S_{t}$ as before.
Thus by change of variable $y=\mtc \sinh\mu$, $S_t$ as $t\rightarrow 0^-$
converges to the
varifold $S_0$ defined by

\[S_0(y,\theta)=(|y|\sqq\,\ept, iy\sqp\,\eqt), y\in \R, 0\leq \theta <2\pi.\]
Moreover, the  identity (\ref{mainid}) holds.

For $t>0$, we define $V_{t}$ to be $e^{i\arg \mu} E_{t}(\mu,
\theta+\f{\arg \mu}{q})$ (note that $\arg \mu=0$ or $\pi$). By
change of variable $y=\ptc \sinh\mu$, it is not hard to see $V_t$
as $t\rightarrow 0^+$ converges to the varifold $V_0$ defined by
\[V_0(y,\theta)=\begin{cases}\begin{aligned}&(y\sqq\,\ept, iy\sqp\,\eqt),& y\geq 0, 0\leq \theta <2\pi\\&
-(y\sqq\,e^{ip(\theta+\f{\pi}{q})}, i|y|\sqp\,e^{-iq(\theta+\f{\pi}{q})}),&
y< 0, 0\leq \theta <2\pi \end{aligned}\end{cases}  \]
$V_0$ coincides with $S_0$ by the change of variable

\begin{equation}\label{cov2} V_{0}(y, \theta)=S_{0} (y, \theta+ \f{\arg y}{q}).\end{equation}

The angle shift of $V_t$ for $\mu<0$ is to make the
parametrization continuous at $\mu=0$. Although the tangent plane
from $\mu\rightarrow 0^+$ and $\mu\rightarrow 0^-$ do not agree.
The induced volume form and mean curvature vector from both sides
are the same. Hence $V_t$ can still be considered as a
self-expander for $t>0$. In fact, the image of $V_t$ for $t>0$ can
be regarded as two complete non-oriented smooth surfaces
intersecting at one circle. We claim that when $V_t$ is considered
as a Radon measure, its generalized mean curvature vector $h(V_t)$
is the same as the usual mean curvature vector. That is, there is
no contribution from the singular set $\{\mu=0\}$. To compute the
generalized mean curvature vector, we choose a family of ambient
diffeomorphism $\psi_s$ with  $\psi_0=id$ and
$\frac{d\psi_s}{ds}|_{s=0}=W$ and derive the first variation
formula $\frac{d|\!|(\psi_s)_*(V_t)|\!|}{ds}|_{s=0}$.

We can divide the image into $\mu<0$ and $\mu>0$ with boundary
curve $\{\ptc(0, i\sqrt{p} e^{-iq\theta}):0\leq \theta<2\pi\}$,
and calculate separately. To prove the claim, the essential part
is to compute the contribution from the boundary. Note that the
unit normal of the boundary from the $\mu>0$ side is
$(e^{ip\theta}, 0)$, while it is $(-e^{ip(\theta+\f{\pi}{q})}, 0)$
from the $\mu<0$ side. We observe that each of $\theta,
\theta+\f{2\pi}{q}, \cdots, \theta+\f{2\pi(q-1)}{q}$ determines
the same boundary point for $0\leq \theta<\f{2\pi}{q}$. Thus the
contribution of the boundary to the first variation from the
$\mu>0$ side is

\[\begin{split}&\int_0^{2\pi} W\cdot (e^{ip\theta}, 0)\ptc \sqrt{p}d\theta\\
&=\ptc \sqrt{p}\int_0^{\f{2\pi}{q}} W\cdot (e^{ip\theta},
0)(1+e^{i\f{2\pi p}{q}}+\cdots+ e^{i\f{2\pi p(q-1)}{q}}
)d\theta=0.\end{split}\]

The last equality follows from the fact that $1+e^{i\f{2\pi
p}{q}}+\cdots+ e^{i\f{2\pi p(q-1)}{q}}=0$ for $p$, $q$ being two
co-prime integers and $q>1$. Since $q$ is a positive even number,
this is certainly the case. The contribution of the boundary from
the $\mu<0$ side is also zero for the same reason. Thus the usual
mean curvature vector agrees with the generalized mean curvature
vector for $V_t$. We have

\[|h(V_t)|^2=\f{c}{t} \f{(p-q)^2}{pq}(p\cosh^2\mu+q\sinh^2\mu)\]
which is bounded for any fixed $t>0$. Thus by the dominate
convergence theorem, we still have

\begin{equation}\label{match}\f{d}{dt}|\!|V_t|\!|(\phi)=-\int \phi \,|h(V_t)|^{2}
d|\!|V_t|\!| +\int D\phi \cdot h(V_t) d|\!|V_t|\!|\end{equation}
for $t>0$.

Hence to apply Proposition 3.1, it suffices to show
\begin{equation}\label{mainid22}\begin{split}&\lim_{t\rightarrow 0^+}-\int \phi \,|h(V_t)|^{2}
d|\!|V_t|\!|
+\int D\phi \cdot h(V_t) d|\!|V_t|\!|\\
&=-\int \phi \,|h(V_0)|^{2}  d|\!|V_0|\!| +\int D\phi \cdot h(V_0)
 d|\!|V_0|\!|. \end{split}\end{equation} This identity can be
 checked similarly using the following equations:

\begin{equation}\label{etnorm2}\left|V_t\right|^2=
\left(\f{t}{c}\right)(p\cosh^2\mu+q\sinh^2\mu),\end{equation}
\begin{equation}\label{hetnormsq2}\left|h(V_t)\right|^2
=\left(\f{c}{t}\right)\f{(p-q)^2}{pq}\f{1}{q\cosh^2\mu+p\sinh^2\mu},\end{equation}

\begin{equation}\label{dett2}d|\!|V_t|\!|=
\left(\f{t}{c}\right)\sqrt{pq}(q\cosh^2\mu+p\sinh^2\mu)d\mu
d\theta.\end{equation}

\subsection{ Case 3:  $p$ even and $q$ odd}
In this case, for $t>0$, $V_{t}$ is defined to be $E_{t}$ as in case 1.
Thus by change of variable $y=\ptc \sinh\mu$, $E_t$ as $t\rightarrow 0^+$
converges to the
varifold $E_0$ defined by
\[E_0(y,\theta)= (y \sqq e^{ip\theta} \,,\,i|y| \sqp
e^{-iq\theta}), y\in \R, 0\leq \theta<2\pi.\]
Moreover, the  identity (\ref{mainid2}) holds.

For $t<0$, we define $V_{t}$ to be $e^{i\arg \mu} S_{t}(\mu,
\theta+\f{\arg \mu}{p})$ (note that $\arg \mu=0$ or $\pi$). By
change of variable $y=\mtc \sinh\mu$, it is not hard to see $V_t$
as $t\rightarrow 0^-$ converges to the varifold $V_0$ defined by
$$V_0(y,\theta)=\begin{cases}\begin{aligned} &(y\sqq\,\ept, iy\sqp\,\eqt), &y\geq 0, 0\leq \theta <2\pi\\
&-(|y|\sqq\,e^{ip(\theta+\f{\pi}{p})}, iy\sqp\,e^{-iq(\theta+\f{\pi}{p})}),&
y< 0, 0\leq \theta <2\pi  \end{aligned}\end{cases}$$
$V_0$ coincides with $E_0$ by the change of variable

\begin{equation}\label{cov3} V_{0}(y, \theta)=E_{0} (y, \theta+ \f{\arg y}{p}).\end{equation}
By similar discussions as in case 2, it can be shown that for
$t<0$, $V_t$  is still a self-shrinker and satisfies
(\ref{match}).

Moreover, for $t<0$, we have $V_t$ with
\begin{equation}\label{stnorm3}\left|V_t\right|^2
=\left(\f{-t}{c}\right)(q\cosh^2\mu+p\sinh^2\mu),\end{equation}
\begin{equation}\label{hstnormsq3}\left|h(V_t)\right|^2
=\left(\f{c}{-t}\right)\f{(p-q)^2}{pq}\f{1}{p\cosh^2\mu+q\sinh^2\mu},\end{equation}

\begin{equation}\label{dst3}d|\!|V_t|\!|
=\left(\f{-t}{c}\right)\sqrt{pq}(p\cosh^2\mu+q\sinh^2\mu)d\mu
d\theta.\end{equation}
The identity
\begin{equation}\label{mainid13}\begin{split}&\lim_{t\rightarrow 0^-}-\int \phi \,|h(V_t)|^{2}
d|\!|V_t|\!|
+\int D\phi \cdot h(V_t) d|\!|V_t|\!|\\
&=-\int \phi \,|h(V_0)|^{2}  d|\!|V_0|\!| +\int D\phi \cdot h(V_0)
 d|\!|V_0|\!|\end{split}\end{equation} can be checked similarly.

 \section{Proof of Theorem 1.2}
 It is important to determine which cone constructed in \cite{sw}
 is area minimizing among Lagrangian competitors. Only these cones can occur as blow-up profiles for the
 singularities in the Lagrangian minimizers. Schoen and Wolfson show that
 a $(p,q)$ cone is stable if and only if $p-q=1$ and conjectured  that only
 $(2,1)$ cone (assuming $p>q$) is area minimizing.
 From the proof in section~4.2, we can in fact show that this is the
 case infinitesimally. In the following, we prove Theorem 1.2.
\bigskip

\begin{proof} Suppose $p>q>1$ and $\mu \geq 0$, $0 \leq \theta <2\pi$. Define
$V_t=S_{t}(\mu, \theta)$ for $t<0$,  $V_t=E_{t}(\mu, \theta)$ and
$V_{0}=C_{pq}$.
Both $S_{t}$ and $E_{t}$ converge to $C_{pq}$ as $t\rightarrow 0$.
When $t \neq 0$, the image of
$\mu=0$ is the boundary of $V_t$. Because both $p$ and $q$ are greater than one,  there is
no boundary contribution on the generalized mean curvature vector. We
take the case $t>0$ as an example. The boundary curve is
$\{\ptc(0, i\sqrt{p} e^{-iq\theta}):0\leq \theta<2\pi\}$ and the
unit normal vector  is $(e^{ip\theta}, 0)$. We observe that each
of $\theta, \theta+\f{2\pi}{q}, \cdots, \theta+\f{2\pi(q-1)}{q}$
determines the same boundary point for $0\leq \theta<\f{2\pi}{q}$.
Thus the contribution from the boundary to the first variation
 is

\[\begin{split}&\int_0^{2\pi} W\cdot (e^{ip\theta}, 0)\ptc \sqrt{p}d\theta\\
&=\ptc \sqrt{p}\int_0^{\f{2\pi}{q}} W\cdot (e^{ip\theta},
0)(1+e^{i\f{2\pi p}{q}}+\cdots+ e^{i\f{2\pi p(q-1)}{q}}
)d\theta=0\end{split}\]if $q$ is an integer greater than one. Because $p>1$, there is no
contribution from the boundary to the first variation when
$t<0$ either. The same arguments as in last section show that
$V_{t}$ forms a solution of the Brakke motion.
\end{proof}

As a generalization of the mean curvature flow, the Brakke motion decreases area.
Theorem 1.2 thus suggests that the cone $C_{p,q}$ may not be area
minimizing when $q>1$. Wolfson's
counterexample \cite{wo} shows that one of the $(p, q)$ cone must be area
minimizing. This leaves the $(2,1)$ cone as the only candidate for area
minimizer. However, we remark this observation does not resolve the
conjecture  because, in classical sense, one needs to find
Lagrangian competitors with the same boundary.

\section{Higher dimensional examples}
 In the two-dimensional case, after multiplying by the matrix
 $\begin{bmatrix} 1&0\\0&-i\end{bmatrix}\in U(2)$, our example can be
 rewritten in the form

 \[\{(x_1 e^{ip\theta}, x_2 e^{-iq\theta})\,|\,\, px_1^2-qx_2^2=pq, (x_1, x_2)\in \R^2 ,0\leq \theta<2\pi
 \}\subset \C^2 .\]

 Now consider the following generalization to higher dimensions:
 for any $n$ nonzero real numbers $\lambda_1,\cdots, \lambda_n$, consider
 the submanifold $\Sigma$ of $\C^n$ defined by

\[\{(x_1 e^{i\lambda_1\theta}, \cdots , x_n e^{i\lambda_n\theta})\,|\,\, \sum_{i=1}^n \lambda_ix_i^2=C,
(x_1, \cdots, x_n)\in \R^n\}\] for some constant $C$.

  It is not hard to check that $\Sigma$ is
 Lagrangian in $\C^n$ with Lagrangian angle given by
 $\beta=(\sum_{i=1}^n \lambda_i)\theta+c$ for some constant $c$ . Therefore $\Sigma$
 is Hamiltonian stationary and is special if $\sum_{i=1}^n\lambda_i=0$.
 Such special Lagrangians were studied by Haskins in \cite{ha1} \cite{ha2} (for $n=3$) and
 Joyce in \cite{jo1} (for general
 dimensions).
 We were also informed by Professor Joyce that the Hamiltonian
 stationary ones
 may also be obtained by applying his method
 of ``perpendicular symmetries " in \cite{jo2}.

If $\sum_{i=1}^n\lambda_i\neq 0$, the position vector $F$
 satisfies
\[F^\perp=\f{-C}{\sum_{i=1}^n \lambda_i}H.\] That is,  the submanifold $\Sigma$
is a Hamiltonian stationary self-similar solution of the mean
curvature flow. Similar procedures as in this paper can be applied
to show that when $\lambda_i$ are all integers,

\[\Sigma_t=\{(x_1 e^{i\lambda_1\theta}, \cdots , x_n e^{i\lambda_n\theta})| \sum_{i=1}^n \lambda_ix_i^2=(-2t)\sum_{i=1}^n \lambda_i,
(x_1, \cdots, x_n)\in \R^n,0\leq \theta<2\pi\}\] form a Brakke
flow without mass loss. We shall discuss further properties of
these higher dimensional examples in a forthcoming paper
\cite{lw}.


\begin{thebibliography}{999}

\bibitem[AN]{an}H. Anciaux, \textit{Construction of Lagrangian self-similar solutions to
the mean curvature flow in $\C^ n$.}  Geom. Dedicata  \textbf{120}
(2006), 37--48.

\bibitem[BR]{br} K.A. Brakke, \textit{The motion of a surface by its mean curvature.} Mathematical Notes,
Princeton University Press, 1978.




\bibitem[FIK]{fik} M. Feldman; T. Ilmanen; D. Knopf, \textit{Rotationally symmetric shrinking and
expanding gradient K\"ahler-Ricci solitons.}  J. Differential Geom.  \textbf{65}  (2003),  no. 2, 169--209.
\bibitem[HL]{hl} R. Harvey; H.B. Lawson, \textit{Calibrated geometries.} Acta Math.
\textbf{148} (1982), 48-156.

\bibitem[HA1]{ha1} M. Haskins, \textit{Special Lagrangian cones.} Amer. J. Math. \textbf{126} (2004), no. 4,
845--871.
\bibitem[HA2]{ha2} M. Haskins, \textit{The geometric complexity of special Lagrangian
$T\sp 2$-cones.} Invent. Math. \textbf{157} (2004), no. 1, 11--70.
\bibitem[HA3]{ha3} M. Haskins, \textit{personal communications}.


\bibitem[HU]{hu} G. Huisken, \textit{Asymptotic behavior for singularities of the mean curvature flow.}
 J. Differential Geom. \textbf{31}  (1990),  no. 1, 285--299.




 \bibitem[JO1]{jo1} D. Joyce, \textit{Constructing special Lagrangian $m$-folds in
 $\Bbb C\sp m$ by evolving quadrics.} Math. Ann. \textbf{320} (2001), no. 4, 757--797.

\bibitem[JO2]{jo2} D.  Joyce, \textit{Special Lagrangian $m$-folds in $\Bbb C\sp m$
with symmetries.} Duke Math. J. \textbf{115} (2002), no. 1, 1--51.

\bibitem[LW]{lw} Y.I. Lee; M.T. Wang, in preparation.
\bibitem[SW]{sw} R. Schoen; J.G. Wolfson, \textit{Minimizing area among Lagrangian surfaces: the mapping
problem.} J. Differential Geom. \textbf{58} (2001), no. 1, 1-86.


\bibitem[SM]{sm} K. Smoczyk, \textit{Angle theorems for the Lagrangian mean curvature
flow.} Math. Z. \textbf{240} (2002), no. 4, 849--883.

\bibitem[SW]{sw} K. Smoczyk; M.-T. Wang, \textit{Mean curvature flows of Lagrangians
submanifolds with convex potentials.} J. Differential Geom.
\textbf{62} (2002), no. 2, 243--257.

\bibitem[SYZ]{syz} A. Strominger; S.-T. Yau; E. Zaslow, \textit{Mirror symmetry is $T$-duality.}
Nuclear Phys. B  \textbf{479}  (1996),  no. 1-2, 243--259.

\bibitem[WA1]{wa1} M.-T. Wang, \textit{Deforming area preserving
diffeomorphism of surfaces by mean curvature flow.} Math. Res.
Lett. \textbf{8} (2001), no.5-6, 651-662.

\bibitem[WA2]{wa2}  M.-T. Wang, \textit{Long-time existence and convergence of graphic mean
curvature flow in arbitrary codimension.} Invent. Math.
\textbf{148} (2002), no. 3, 525--543.


\bibitem[WO]{wo} J.G. Wolfson, \textit{Lagrangian homology classes without regular
minimizers}, J. Differential Geom. \textbf{71} (2005), 307-313.


 \end{thebibliography}
\end{document}